\documentclass{amsart}
\usepackage[sorted-cites]{amsrefs}
\usepackage{graphicx}
\usepackage{latexsym}
\usepackage{amsthm}

\usepackage{amsfonts}
\usepackage{xcolor}
\usepackage[all]{xy}

\usepackage{amssymb, mathrsfs, amsfonts, amsmath}
\usepackage{amsbsy}
\usepackage{amsfonts}
\newtheorem{corollary}{Corollary}

\newtheorem{lemma}{Lemma}

\newtheorem{remark}{Remark}
\newtheorem{theorem}{Theorem}
\newtheorem{example}{Example}
\numberwithin{equation}{section}

\global\long\def\Vol{\mathrm{Vol}}

\begin{document}
\title[Volume growth estimates]{Volume growth estimates for Ricci solitons\\ and quasi-Einstein manifolds}

\author{Xu Cheng}
\author{Ernani Ribeiro Jr}
\author{Detang Zhou}

\address[X. Cheng]{Instituto de Matem\'atica e Estat\'istica, Universidade Federal Fluminense - UFF, 24020-140, Niter\'oi - RJ, Brazil}
\email{xucheng@id.uff.br}

\address[E. Ribeiro Jr]{Departamento  de Matem\'atica, Universidade Federal do Cear\'a - UFC, Campus do Pici, Av. Humberto Monte, Bloco 914,
60455-760, Fortaleza - CE, Brazil}\email{ernani@mat.ufc.br}

\address[D. Zhou]{Instituto de Matem\'atica e Estat\'istica, Universidade Federal Fluminense - UFF, 24020-140, Niter\'oi - RJ, Brazil}
\email{zhou@impa.br}

\thanks{X. Cheng and D. Zhou were partially supported by CNPq/Brazil and FAPERJ/ Brazil - Finance Code 001}

\thanks{E. Ribeiro was partially supported by CNPq/Brazil [Grant: 305410/2018-0 \& 160002/2019-2] and CAPES/ Brazil - Finance Code 001}



\date{\today}

\newcommand{\spacing}[1]{\renewcommand{\baselinestretch}{#1}\large\normalsize}
\spacing{1.2}

\begin{abstract}

In this article, we provide some volume growth estimates for complete noncompact gradient Ricci solitons and quasi-Einstein manifolds similar to the classical results by Bishop, Calabi and Yau for complete Riemannian manifolds with nonnegative Ricci curvature. We prove a sharp volume growth estimate for complete noncompact gradient shrinking Ricci soliton. Moreover, we provide upper bound volume growth estimates for complete noncompact quasi-Einstein manifolds with $\lambda=0.$ In addition, we prove that geodesic balls of complete noncompact quasi-Einstein manifolds with $\lambda<0$ and $\mu\leq 0$ have at most exponential volume growth. 
\end{abstract}

\maketitle
\section{Introduction}

In this article, we deal with two objects: gradient shrinking Ricci solitons and quasi-Einstein manifolds.  The first part is on the volume growth of gradient shrinking Ricci solitons. Recall that a  complete Riemannian metric $g$ on a smooth $n$-dimensional manifold $M^n$ is called {\it gradient shrinking Ricci soliton} if there exists a smooth potential function $f$ on $M^n$ such that the Ricci tensor $Ric$ of the metric $g$ satisfies the following equation
\begin{equation}
\label{maineq-0}
Ric+\nabla^2 f=\lambda g,
\end{equation} for some positive  constant $\lambda.$ Here, $\nabla^2\,f$ denotes the Hessian of $f.$
Without loss of generality, we  take  $\lambda=\frac{1}{2}$ in (\ref{maineq-0}), that is, $g$ satisfies
\begin{equation}
\label{maineq}
Ric+\nabla^2f=\frac{1}{2} g.
\end{equation} 
The potential function  $f$ can be normalized, by adding a suitable constant
to it,  to satisfy the following equation
\begin{equation}\label{eq-f}
R+|\nabla f|^{2}=f,
\end{equation} where $R$ is the scalar curvature of $M^n.$

Gradient Ricci solitons are important in understanding the Hamilton's Ricci flow \cite{Hamilton2}. They arise often as singularity models of the Ricci flow. It has been confirmed by Enders, M\"uller and Topping \cite{Topping} that the blow-ups around a type-I singularity point of a Ricci flow  converge to (nontrivial) gradient shrinking Ricci solitons, see also \cite{Naber,Sesum}. In view of their importance, it is  natural to seek classification results for gradient shrinking Ricci solitons. We refer the readers to the survey \cite{caoALM11} and references therein for a nice overview on the subject.  

On the other hand, a good knowledge of the volume growth rate is one of basic  geometric informations on which various other properties of the underlying Riemannian manifold are built. We now briefly recall a few relevant results on volume growth estimate. A theorem due to Calabi \cite{calabi} and Yau  \cite{yau1} asserts that the geodesic balls of complete noncompact manifolds with nonnegative Ricci tensor have at least linear volume growth, that is, $$\Vol (B_{p}(r))\geq cr,$$ for any $r>r_{0},$ where $r_{0}$ is a positive constant and $B_{p}(r)$ is the geodesic ball of radius $r$ centered at $p\in M^n$ and $c$ is a constant that does not depend on $r.$ Similar results were obtained by Munteanu and Sesum \cite{Natasa} on gradient shrinking Ricci solitons.

The classical Bishop volume comparison theorem asserts that  for a complete noncompact $n$-dimensional Riemannian manifold with nonnegative Ricci tensor,  the volume of the geodesic balls of radius $r$  are no more than the one of the balls of the radius $r$ in the Euclidean space $\mathbb{R}^n$ and hence it must have at most polynomial volume growth. In \cite{Cao}, Cao and Zhou proved that gradient shrinking Ricci solitons have at most Euclidean volume growth, which is an analog of Bishop's theorem for gradient shrinking Ricci solitons (for steady case, see \cite{Natasa}). More precisely, they showed that  for a gradient shrinking Ricci soliton,  its  geodesic balls of radius $r$ centered at a fixed point $p$ have the volume
\begin{equation}
\label{eq1v}
\Vol(B_{p}(r))\le cr^{n},
\end{equation} for some positive constant $c$ and $r>0$ sufficiently large, where  the constant $c$  depends on the geometry of the unit ball at $p$. The Gaussian shrinking soliton $\Big(\Bbb{R}^{n},\,\delta_{ij},\,f(x)=\frac{|x|^{2}}{4}\Big)$ guarantees that the Euclidean growth rate proved by Cao and Zhou is optimal. In this article, we will prove sharp volume growth upper bounds for complete noncompact gradient shrinking Ricci solitons in the sense that they can be achieved by the Gaussian shrinking soliton.  Before stating the results, let us give some notation. Fix a point $p$ on $M^n$. In terms of polar normal coordinates at $p,$ we may write the volume element as $J^{n-1}(\theta,r)dr\wedge d\theta,$ where $d\theta$ is the area element of the unit $(n-1)$-dimensional sphere $\Bbb{S}^{n-1}.$ In particular, the Gauss lemma asserts that the area element of the boundary of the geodesic ball of radius $r$ is given by $J^{n-1}(\theta,r)d\theta$.

Now, we are ready to state our first result.

\begin{theorem}\label{thm-1}
 Let $(M^{n},\,g,\,f)$ be an $n$-dimensional complete noncompact gradient shrinking Ricci soliton with $f$ satisfying (\ref{eq-f}). Let $p\in M$ be a fixed point. Then for all $r>0,$ the volume of the geodesic ball $B_p(r)$ of radius $r$ centered at $p$ satisfies
\begin{equation}
\label{ine-vol1-c}
\Vol(B_{p}(r))\le \int_{\Bbb{S}^{n-1}}\int_{0}^{r}e^{f(p)-\frac{1}{r}\int_{0}^{r}R(\theta,s)ds}r^{n-1}drd\theta.
\end{equation}  Furthermore, the equality in (\ref{ine-vol1-c}) holds for all $r$ if and only if $(M^{n},\,g,\,f)$ is  $(\mathbb{R}^{n},\,\delta_{ij},\,\frac{|x-p|^{2}}{4})$, i.e., a Gaussian shrinking soliton $(\mathbb{R}^{n},\,\delta_{ij},\,\frac{|x|^{2}}{4})$ up to a translation.
\end{theorem}

We remark that the volume growth upper bound obtained in (\ref{ine-vol1})  does not depend of the geometry of the unit ball. Moreover,  the  estimate holds for all $r>0$ not necessarily (sufficiently) large. 
 As a consequence of Theorem \ref{thm-1} we have the following result.

\begin{corollary}\label{cor-1}
 Let $(M^{n},\,g,\,f)$ be a complete noncompact
gradient shrinking Ricci soliton with $f$ satisfying (\ref{eq-f}). Let $p\in M$ be a fixed point.
Then for all $r>0$,  the volume of the geodesic ball $B_p(r)$ of radius $r$ centered at $p$ satisfies
\begin{equation}\label{ine-vol2}
\Vol(B_{p}(r))\le e^{f(p)-\inf_MR}\omega_nr^n
\end{equation}
and 
\begin{equation}\label{ine-vol2-c}
\Vol(B_{p}(r))\le e^{f(p)}\omega_nr^n,
\end{equation}
where  $\omega_n$ denotes the volume of the unit Euclidean ball. Moreover, the equality in (\ref{ine-vol2}) holds for all $r$  if and only if $(M^{n},\,g,\,f)$
is $(\mathbb{R}^{n},\,\delta_{ij},\,\frac{|x-p|^{2}}{4})$, i.e., a Gaussian shrinking soliton $(\mathbb{R}^{n},\,\delta_{ij},\,\frac{|x|^{2}}{4})$ up to  a translation. 
\end{corollary}

A relevant observation is that the inequalities obtained in Theorem \ref{thm-1} and Corollary \ref{cor-1}  are optimal in the sense that the equalities  hold for  the Gaussian shrinking soliton $(\mathbb{R}^{n},\,\delta_{ij},\,\frac{|x|^{2}}{4})$ up to a suitable  translation and conversely, these equalities imply that the Ricci soliton must be the Gaussian shrinking soliton $(\mathbb{R}^{n},\,\delta_{ij},\,\frac{|x|^{2}}{4})$ up to a translation.

We also remark that the inequalities obtained in Theorem \ref{thm-1} and Corollary \ref{cor-1} may be  generalized to the case of complete smooth metric measure space $(M^n,\, g,\, e^{-f}dv)$ satisfying $Ric_f\geq \frac12g$ and $|\nabla f|^2\leq f$, where $Ric_f=Ric+\nabla^2f$ stands for the Bakry-\'Emery curvature tensor; for more details, see Theorem \ref{thm-1-f} in Section \ref{soliton}. In this context, Munteanu and Wang \cite[Theorem 1.4]{MW2} showed that if an $n$-dimensional complete smooth metric measure space $(M^{n},\,g,\,e^{-f}dv)$ satisfies $Ric_f\geq \frac12g$ and $|\nabla f|^2\leq f,$ then $(M^{n},\,g)$ has Euclidean volume growth, i.e., $\text{Vol}(B_p(r))\leq c(n)e^{f(p)}r^n,$ for all $r>0$ and $p\in M.$ In particular, the constant $c(n)$ depends only on $n.$ Besides, in the special case of gradient shrinking Ricci solitons, if $p$ is chosen to be a minimum point of $f,$ then it is easy to check that $f(p)\leq \frac{n}{2}$ and hence $\text{Vol}(B_p(r))\leq c_0r^n$ for $r>0$,  where the constant $c_0=c(n)e^{\frac{n}{2}}$ depends only on $n$. 
This special case was observed by Haslhofer and M\"uller \cite{HM} by proving the existence of a constant $c_0$ depending only on $n$. 

It should be emphasized that, in \eqref{ine-vol2-c} of Corollary \ref{cor-1} and \eqref{ine-vol1-cc} of Theorem \ref{thm-1-f}, we obtain the explicit value of the constant $c(n)$,  i.e., $c(n)=\omega_n$. Hence, in the case of gradient shrinking Ricci solitons, if $p$ is a minimum point of $f,$ then it follows from \eqref{ine-vol2-c} that  $\Vol(B_{p}(r))\le c_0r^n$ for all $r>0$, where the constant $c_0=\omega_n e^{\frac{n}{2}}$ depends only on $n$.

\vspace{0.25cm}

In the second part of the article we discuss volume growth estimates for quasi-Einstein manifolds. Recall that, according to \cite{CaseShuWey}, a complete Riemannian manifold $(M^n,\,g),$ $n\geq 2,$ is an $m$-{\it quasi-Einstein manifold}, or simply {\it quasi-Einstein manifold}, if there exists a smooth potential function $f$ on $M^n$ satisfying the following fundamental equation
\begin{equation}
\label{eqqem}
Ric_{f}^{m}=Ric+\nabla ^2f-\frac{1}{m}df\otimes df=\lambda g,
\end{equation} for some constants $\lambda$ and $m\neq 0.$ It is known that, on a quasi-Einstein manifold, there is a constant  $\mu$ such that 
\begin{equation}\label{2eq}
\Delta_{f} f = m\lambda-m\mu e^{\frac{2}{m}f},
\end{equation} where $\Delta_{f}=\Delta -\langle\nabla f,\,\cdot\,\,\rangle$ is the $f$-Laplacian. For more details on (\ref{2eq}), we refer the readers to \cite{KK}.

We say that a quasi-Einstein manifold is \emph{trivial} if its potential function $f$ is constant, otherwise, we say that it is \emph{nontrivial}. Hence, the triviality implies that $M^n$ is an Einstein manifold. An $\infty$-quasi-Einstein manifold is a gradient Ricci soliton. We also remark that $1$-quasi-Einstein manifolds are more commonly called {\it static metrics} and such metrics have connections to the prescribed scalar curvature pro\-blem, the positive mass theorem and general relativity. As discussed by Besse \cite[pg. 265]{Besse}, an $m$-quasi-Einstein manifold corresponds to a base of a warped product Einstein metric; for more details see Corollary 9.107 in \cite[pg. 267]{Besse} (see also Theorem 1 in \cite{Ernani2}). Another interesting motivation comes from the study of diffusion operators by Bakry and \'Emery \cite{BE}.

Nontrivial examples of quasi-Einstein manifolds can be found, for instance, in \cite{Besse,HPW,LuePage,Wang}. It is also important to highlight that Case \cite{Case} showed that complete $m$-quasi-Einstein manifolds with $\lambda=0$ and $\mu\leq0$ are trivial. While Qian \cite{Qian} proved that complete $m$-quasi-Einstein manifolds with $\lambda>0$ must be compact. Moreover, by Kim and Kim \cite{KK} nontrivial compact quasi-Einstein manifolds must have $\lambda>0.$ Therefore, it follows that a complete  nontrivial quasi-Einstein manifold is compact if and only if $\lambda>0$ (see also \cite[Theorem 4.1]{HPW}). In this article, we focus on complete noncompact quasi-Einstein manifolds.  Consequently, $\lambda$ must be nonpositive. 

In what follows, we recall some examples; see \cite[Table 2]{HPW}.

\begin{example}
\label{example2} Let $\Bbb{H}^n$ be the hyperbolic space form of constant sectional curvature $-1$ with metric $g_{\Bbb{H}^n}=dt^{2}+\sinh^{2}tg_{\Bbb{S}^{n-1}}$ and potential function $f(t)=C\log (\cosh t),$ where $C$ is a constant. Thus, it is a noncompact quasi-Einstein manifold with $\lambda<0$ and $\mu<0.$
\end{example}

\begin{example}
\label{example2a} Let $M=[0,\,\infty)\times F$ with metric $g=dt^{2}+g_{F},$ where $g_{F}$ is a Ricci flat metric, and potential function $f=-m \log c\,t,$ where $c$ is an arbitrary positive constant. So, $M$ is a noncompact quasi-Einstein manifold with $\lambda=0$ and $\mu>0.$
\end{example} 

Another interesting example of noncompact quasi-Einstein manifold with $\lambda=0$ is the Generalized Schwarzschild metric; for more details, we refer the reader to \cite[Example 9.118(a)]{Besse} or \cite[Example 2]{HPW}.

In the sequel, by adapting the method used in the proof of Theorem \ref{thm-1}, we provide an upper bound volume growth estimate for complete quasi-Einstein manifolds with $\lambda=0.$ To be precise, we have established the following theorem.

\begin{theorem}
\label{thmqEmNew}
Let $(M^{n},\,g,\,f)$ be a complete noncompact $m$-quasi-Einstein manifold with $\lambda=0.$ Then for all $r>0,$ the volume of the geodesic ball $B_{p}(r)$ of radius $r$ centered at a point $p$ satisfies

\begin{equation}
\Vol(B_{p}(r))\le \int_{\Bbb{S}^{n-1}}\int_{0}^{r}e^{\Phi}r^{n-1}dr d\theta,
\end{equation} where $\Phi=f(\theta, r)+f(p)-\frac{2}{r}\int_{0}^{r}f(\theta, s)ds$.

\end{theorem}

Our next result is an upper bound weighted volume growth for geodesic balls of quasi-Einstein manifolds with $\lambda=0.$ Here, the weighted measure is given by $ dV_{f}=e^{-f}dV$ and $dV$ denotes the Riemannian measure of $(M^{n},\,g).$ More precisely, we get the following result.

\begin{theorem}\label{thmA}
Let $\big(M^{n},\,g,\,f\big)$ be a complete noncompact quasi-Einstein manifold with $\lambda=0.$ Then there exist positive constants $b$ and $c$ so that
 $$\Vol_{f}(B_p(r))\leq be^{c\,r},$$ for any $r>0$ sufficiently large. 
 \end{theorem}

This result implies in particular that every complete noncompact quasi-Einstein manifold with $\lambda=0$ is $f$-stochastically complete. Here, $f$-stochastic completeness stands for the stochastic completeness of the diffusion process associated to the $f$-Laplacian operator $\Delta_{f}.$ For more details on stochastically completeness see, for instance, \cite{Grigo}. Another consequence of Theorem \ref{thmA} is an upper bound estimate for the first eigenvalue $\lambda_{1}(\Delta_{f})$ of the $f$-Laplacian $\Delta_{f}$. To be precise, we have the following corollary.

\begin{corollary}
\label{cor1}
Let $\big(M^{n},\,g,\,f\big)$ be a complete noncompact quasi-Einstein manifold with $\lambda=0.$ Then, it holds that $$\lambda_{1}(\Delta_{f})\leq \frac{c^{2}}{4},$$ where the constant  $c$ is the same as in Theorem \ref{thmA}. 
\end{corollary}

In \cite{BRR2019}, Batista, Ranieri and Ribeiro obtained  volume growth  lower bounds for geodesic balls of complete noncompact $m$-quasi-Einstein manifolds with $\lambda<0.$  Here, we will provide a  volume growth upper bound for geodesic balls of complete noncompact quasi-Einstein manifolds with $\lambda<0$ and $\mu\leq 0.$ More precisely, we will  establish the following result.

\begin{theorem}\label{thm-2}
	Let $(M^{n},\,g,\,f)$ be a complete noncompact $m$-quasi-Einstein manifold  with $\lambda<0,$ $\mu\leq 0$ and $m\in(1,\infty).$ Then there exist positive constants $a$ and $b$ such that
	
 \begin{equation}
 \label{e1ad}
	\Vol(B_{p}(r))\leq a e^{b\,r},
	\end{equation} for any $r>0$ sufficiently large.

\end{theorem}

Observe that  Example \ref{example2} has exponential volume growth. Therefore,  the growth rate obtained in (\ref{e1ad}) of Theorem \ref{thm-2} is optimal. However, the constant $a$ depends of the volume of unit ball.

\section{Volume growth of gradient shrinking Ricci solitons}\label{soliton}

In this section, we shall present the proof of Theorem \ref{thm-1}. Before doing this, let us recall some important features of gradient shrinking Ricci solitons (cf. \cite{Hamilton2}). More precisely, up to normalization of the potential function $f,$ we have the following lemma. 

\begin{lemma}[\cite{Hamilton2}]
\label{lem1}
Let $\big(M^n,\,g,\,f\big)$ be a gradient shrinking Ricci soliton. Then we have:
\begin{enumerate}
\item $R+\Delta f=\dfrac{n}{2}.$
%\item $\frac{1}{2}\nabla R=Ric(\nabla f).$
%\item $\Delta_{f} R= R-2|Ric|^{2}.$
\item $R+|\nabla f|^{2}= f$.
\item $\Delta f-|\nabla f|^{2}+ f=\dfrac{n}{2}$.
%\item $\Delta_{f} R_{ij}=R_{ij}-2R_{ikjl}R_{kl}.$
%\item $\Delta_{f} Rm =Rm+ Rm\ast Rm.$
%\item $\nabla_{l}R_{ijkl}=R_{ijkl}f_{l}=\nabla_{j}R_{ik}-\nabla_{i}R_{jk}.$
\end{enumerate}
\end{lemma}

In \cite{Chen}, Chen showed that every gradient shrinking Ricci soliton has nonnegative scalar curvature. Concerning the potential function $f,$ Cao and Zhou \cite{Cao} proved that 
\begin{equation}
\label{eqfbeh}
\frac{1}{2}\Big(r(x)-c\Big)^{2}\le f(x)\le \frac{1}{2}\Big(r(x)+c\Big)^{2},
\end{equation} for all $r(x)\geq r_{0}.$ This combined with Lemma \ref{lem1} provides an asymptotic behaviour of the scalar curvature.

\vspace{0.3cm}

Now we are ready to prove Theorem \ref{thm-1}.

\vspace{0.3cm}

\subsection{Proof of Theorem \ref{thm-1}}
\begin {proof}
Initially, if $x=(\theta,r)$ is not in the cut-locus $\mathcal{C}(p)$ of $p$, we have
\[
w(\theta,r)=\frac{\partial}{\partial r}\log J(\theta,r)=\frac{\frac{\partial J}{\partial r}}{J}(\theta,r).
\]
In particular, it is well known that
\[
w'(\theta,r)+w^{2}(\theta,r)+\frac{1}{n-1}Ric \Big(\frac{\partial}{\partial r},\frac{\partial}{\partial r}\Big)\le0,
\]
where $w':=\frac{\partial w}{\partial r}$. Multiplying by $r^{2}$ and integrating from $\varepsilon$ to $r$ yields 
\[
\int_{\varepsilon}^{r}s^{2}w'ds+\int_{\varepsilon}^{r}s^{2}w^{2}ds+\int_{\varepsilon}^{r}\frac{1}{n-1} Ric\Big(\frac{\partial}{\partial s},\frac{\partial}{\partial s}\Big)s^{2}ds\le 0.
\]
Let $\varepsilon\to 0$. Then, it follows that
\[
\begin{split}r^{2}w & \le-\int_{0}^{r}(sw-1)^{2}ds+\int_{0}^{r}\left[1-s^{2}\frac{1}{n-1} Ric\Big(\frac{\partial}{\partial s},\frac{\partial}{\partial s}\Big)\right]ds\\
 & \le\int_{0}^{r}\left[1-\frac{s^{2}}{n-1} Ric\Big(\frac{\partial}{\partial s},\frac{\partial}{\partial s}\Big)\right]ds.
\end{split}
\]
Therefore, we obtain 
\begin{equation}\label{eq-6}
\Big(\log\frac{J}{r}\Big)'\le-\frac{1}{r^{2}}\int_{0}^{r}\frac{s^{2}}{n-1} Ric\Big(\frac{\partial}{\partial s},\frac{\partial}{\partial s}\Big)ds.
\end{equation} Now, upon integrating (\ref{eq-6}) from $\varepsilon$ to $r,$ we let $\varepsilon\rightarrow 0$. This yields  
\begin{equation}\label{eq-7}
\begin{split}(n-1)\Big(\log\frac{J}{r}\Big) & \le-\int_{0}^{r}\left[\frac{1}{t^{2}}\int_{0}^{t}s^{2} Ric\Big(\frac{\partial}{\partial s},\frac{\partial}{\partial s}\Big)ds\right]dt\\
 & =\frac{1}{r}\int_{0}^{r}s^{2} Ric\Big(\frac{\partial}{\partial s},\frac{\partial}{\partial s}\Big)ds-\int_{0}^{r}t Ric\Big(\frac{\partial}{\partial t},\frac{\partial}{\partial t}\Big)dt,
\end{split}
\end{equation} where we have used that $\displaystyle\lim_{r\to0}\frac{J}{r}=1$. Thus, combining (\ref{eq-6}) and (\ref{eq-7}) we obtain
\begin{equation}\label{eq-8}
\begin{split}(n-1)\Big(r\log\frac{J}{r}\Big)' & \le-\int_{0}^{r}s Ric\Big(\frac{\partial}{\partial s},\frac{\partial}{\partial s}\Big)ds.\end{split}
\end{equation}

In order to proceed, let $\gamma$ be the minimizing geodesic joining  $x$ from $p$ and denote
$f(s)=f(\gamma(s))$, $0\leq s\leq r.$ The Ricci soliton equation (\ref{maineq}) gives 
\begin{equation}\label{eq-8-1}
Ric \Big(\frac{\partial}{\partial s},\frac{\partial}{\partial s}\Big)(\theta,s)=\frac{1}{2}-\frac{\partial^{2}f}{\partial s^{2}}=\frac{1}{2}-f''(s).
\end{equation} 
Substituting (\ref{eq-8-1}) into (\ref{eq-8}) results in
\begin{eqnarray}
\label{eq-9}
(n-1)\left(r\log\frac{J}{r}\right)' & \le &  -\int_{0}^{r}s\Big(\frac{1}{2}-f''(s)\Big)ds\nonumber\\
 & =&-\frac{r^{2}}{4}+rf'(r)-f(r)+f(0)\nonumber\\
 & =&-\frac{r^{2}}{4}+r\langle\nabla f,\nabla r\rangle(x)-f(x)+f(p)\nonumber\\
 & =&-\frac{r^{2}}{4}+r\langle\nabla f,\nabla r\rangle(x)-R(x)-|\nabla f|^{2}(x)+f(p)\nonumber\\
 & =&-\left(\frac{r}{2}-\langle\nabla f,\nabla r\rangle(x)\right)^{2}-\left(|\nabla f|^{2}(x)-\langle\nabla f,\nabla r\rangle^{2}(x)\right)-R(x)+f(p)\nonumber\\
 & \le&-R(x)+f(p).
\end{eqnarray} Here we have used the equality $f=R+|\nabla f|^{2}.$ Next, integrating  (\ref{eq-9})
from $\varepsilon$ to $r$ and letting $\varepsilon\rightarrow 0$ we obtain
\[
(n-1)\Big(r\log\frac{J}{r}\Big)\le-\int_{0}^{r}R(\theta,s)ds+f(p)r.
\]
Hence, it follows that 
\[
J^{n-1}(\theta,r)\le e^{-\frac{1}{r}\int_{0}^{r}[R(\theta,s)-f(p)]ds}r^{n-1}.
\]
Consequently,
\begin{equation}\label{eq-10}
\begin{split}
\Vol(B_p(r))&=\int_{\Bbb{S}^{n-1}}\int_{0}^{\min\{r,\rho(\theta)\}}J^{n-1}(\theta,r)drd\theta\\
&\le  \int_{\Bbb{S}^{n-1}}\int_{0}^{\min\{r,\rho(\theta)\}}e^{f(p)-\frac{1}{r}\int_{0}^{r}R(\theta,s)ds}r^{n-1}drd\theta\\
&\le  \int_{\Bbb{S}^{n-1}}\int_{0}^{r}e^{f(p)-\frac{1}{r}\int_{0}^{r}R(\theta,s)ds}r^{n-1}drd\theta,
\end{split}
\end{equation} in the above $\rho(\theta)$ denotes the cut-locus radius in the direction
$\theta$. Therefore, our assertion (\ref{ine-vol1}) is proved.

Proceeding, if the equality in (\ref{eq-10}) holds, we may use (\ref{eq-9}) to deduce 
\begin{equation}\nonumber
\left(\frac{r}{2}-\langle\nabla f,\nabla r\rangle(x)\right)^{2}=0 \quad\hbox{and}\quad
|\nabla f|^{2}(x)-\langle\nabla f,\nabla r\rangle^{2}(x)=0.
\end{equation}
This implies that $(\nabla f)(x)=\langle \nabla f, \nabla r\rangle\dfrac{\partial }{\partial r}=\dfrac{r}{2}\dfrac{\partial }{\partial r}$ and then
\begin{equation}\nonumber
f(x)=f(p)+\dfrac{r^2}{4}.
\end{equation}  By the property of the cut-locus and the smoothness of $f,$ the above expressions for $f$ and $\nabla f$  hold for all $x\in M.$
Note that $f$ satisfies the equation $f=R+|\nabla f|^2$, i.e.,  Equation (2) in Lemma \ref{lem1}. Hence, we have
$$R(x)+\dfrac{r^2}{4}=f(x)=f(p)+\dfrac{r^2}{4}.$$ Then $R(x)=f(p)$ is constant.
So, the equality in (\ref{eq-10}) becomes
\begin{equation}\label{ine-vol2-1}
\Vol(B_{p}(r))= \Vol_{\mathbb{R}^{n}}(r).
\end{equation}
However, it was proved in \cite{ChZ2013} and \cite{Zhang} that the volume growth of geodesic spheres is no more than $cr^{n-2\inf_MR}$ for some positive constant $c$.  Comparing this with (\ref{ine-vol2-1}), we obtain  $\displaystyle\inf_M R=0.$ Therefore,  $R\equiv 0$ on $M^n$ and hence, as a gradient shrinking Ricci soliton, $M^n$ must be the Euclidean space $\mathbb{R}^n$ with the standard metric. Moreover, we have $f(p)=0$ and $f(x)=\dfrac{r^2}{4}=\dfrac{|x-p|^2}{4}$. So, $(M^{n},\,g,\, f)$ is the Gaussian shrinking soliton, which finishes the proof of the theorem.

\end{proof}

\vspace{0.3cm}

\subsubsection{Proof of Corollary  \ref{cor-1}} 
\begin{proof} Since the gradient shrinking Ricci soliton $(M^n,\,g,\,f)$ has nonnegative scalar curvature, it suffices to use  \eqref{ine-vol1-c} to infer
\begin{equation}\label{ine-vol2-s}
\Vol(B_{p}(r))\le e^{f(p)-\inf_MR}\Vol_{\mathbb{R}^n}(r),
\end{equation} as asserted.

On the other hand, if the equality holds in \eqref{ine-vol2} (or \eqref{ine-vol2-c}), then the equality also occurs in \eqref{ine-vol1-c} and therefore, the result follows from Theorem \ref{thm-1}. 

\end{proof}

\vspace{0.3cm}

\begin{remark}
We point out that if $p\in M$ is a minimum point of both $f$ and the scalar curvature $R,$ then \eqref{ine-vol2} implies that $\Vol(B_p(r))$ is at most the volume of the Euclidean ball of radius $r$.
\end{remark}

In the sequel, we going to present a generalization of Theorem \ref{thm-1} and Corollary \ref{cor-1} to the case of  complete smooth metric measure spaces $(M^n,\, g,\, e^{-f}dv).$ More precisely, we have the following result.

\begin{theorem}\label{thm-1-f}
 Let $(M^n,\, g,\, e^{-f}dv)$ be an $n$-dimensional  complete smooth metric measure space satisfying $Ric_f\geq \frac12g$ and $|\nabla f|^2\leq f$. Then for all $r>0,$ the volume of the geodesic ball $B_p(r)$ of radius $r$ centered at a point $p$ satisfies
\begin{equation}
\label{ine-vol1}
\Vol(B_{p}(r))\le \int_{\Bbb{S}^{n-1}}\int_{0}^{r}e^{f(p)-\frac{1}{r}\int_{0}^{r}(f-|\nabla f|^2)(\theta,s)ds}r^{n-1}drd\theta
\end{equation}
and 
\begin{equation}
\label{ine-vol1-cc}
\Vol(B_{p}(r))\le e^{f(p)-\inf_{M}(f-|\nabla f|^2)}\omega_nr^n\le e^{f(p)}\omega_nr^n,
\end{equation} where $\omega_n$ denotes the volume of the unit Euclidean ball. 
\end{theorem}
\begin{proof}
Following the same steps of the proof of inequality \eqref{ine-vol1-c}, it suffices to use $Ric+\nabla^2f\geq \frac12g$ instead of $Ric+\nabla^2f= \frac12g$ and $f-|\nabla f|^2$ instead of $R,$ respectively, in order to obtain \eqref{ine-vol1-cc}. So, we omit the details.
\end{proof}

\vspace{0.3cm}

\section{Volume growth of quasi-Einstein manifolds}

In this section, we will present the proof of Theorems \ref{thmqEmNew}, \ref{thmA} and \ref{thm-2}. To do so, we need to recall some basic facts. Firstly, we assume that $\big(M^{n},\,g,\,f\big)$ is an $m$-quasi-Einstein manifold,  $m<\infty$, satisfying Eq. (\ref{eqqem}).  In this case, we may consider the function $u = e^{-\frac{f}{m}}$ on $M^n$  and immediately get 
 
 \begin{equation}
 \label{1a}
 \nabla u = -\frac{u}{m}\nabla f
 \end{equation} as well as

\begin{equation}\label{fu}
\nabla^2 f - \frac{1}{m}df \otimes df = -\frac{m}{u}\nabla^2u.
\end{equation} 
In particular, notice that (\ref{eqqem}) and (\ref{fu}) yield 

\begin{equation}\label{quasiforu}
Ric -\frac{m}{u}\nabla^2u=\lambda g. 
\end{equation} Moreover, tracing (\ref{quasiforu}) we have
\begin{equation}
\label{traceu}
R -\frac{m}{u}\Delta u=\lambda n,
\end{equation} where $R$ denotes the scalar curvature of $M^n.$ Furthermore, taking into account (\ref{eqqem}) and (\ref{2eq}), it is not difficult to show that
\begin{equation}\label{fundamentalequationforu}
\frac{u^2}{m}(R-\lambda n) + (m-1)|\nabla u|^2 = -\lambda u^2 + \mu.
\end{equation} 
\vspace{0.3cm}

Now, we present the proofs of Theorems \ref{thmqEmNew}, \ref{thmA} and \ref{thm-2}.

\vspace{0.3cm}

\subsection{Proof of Theorem \ref{thmqEmNew}}
\begin{proof}
From  (\ref{eq-8})  we have

\begin{equation}
\label{eqK1}
(n-1)\left(r\log \frac{J}{r}\right)' \leq - \int_{0}^{r}s Ric\Big(\frac{\partial}{\partial s},\frac{\partial}{\partial s}\Big)ds.
\end{equation}

On the other hand, since $M^n$ is a quasi-Einstein manifold with $\lambda=0,$ we may use the fundamental equation (\ref{eqqem}) to infer

\begin{equation*}
Ric\Big(\frac{\partial}{\partial s},\frac{\partial}{\partial s}\Big)=-f''(s)+\frac{1}{m}(f'(s))^{2},
\end{equation*} where $f(s)=f(\gamma(s)),$ $0\leq s\leq r,$ and $\gamma$ is the minimizing geodesic joining $x$ from a fixed point $p.$ Therefore, returning to the inequality (\ref{eqK1}) we obtain

\begin{eqnarray*}
(n-1)\left(r\log \frac{J}{r}\right)'  &\leq & -\int_{0}^{r}s\Big[-f''(s)+\frac{1}{m}(f'(s))^{2}\Big]ds\nonumber\\&\leq & \int_{0}^{r}sf''(s)ds\nonumber\\&=& rf'(r)-f(r)+f(0).
\end{eqnarray*} Hence, it follows that

\begin{eqnarray*}
(n-1)\left(r\log \frac{J}{r}\right) \leq \int_{0}^{r}sf'(s)ds - \int_{0}^{r}f(s)ds +f(p)r.
\end{eqnarray*} Consequently,

\begin{eqnarray*}
(n-1)\left(r\log \frac{J}{r}\right) \leq rf(r)+f(p)r-2\int_{0}^{r}f(s)ds
\end{eqnarray*} and therefore, we obtain

\begin{eqnarray*}
J^{n-1}\leq e^{\big(f(r)-f(p)-\frac{2}{r}\int_{0}^{r}f(s)ds\big)}r^{n-1}.
\end{eqnarray*} This allows us to conclude

\begin{eqnarray}
\Vol(B_{p}(r)) &=&  \int_{\Bbb{S}^{n-1}}\int_{0}^{\min\{r,\rho(\theta)\}}J^{n-1}(\theta, r)dr d\theta \nonumber\\&\le & \int_{\Bbb{S}^{n-1}}\int_{0}^{\min\{r,\rho(\theta)\}}e^{\Phi}r^{n-1}dr d\theta\nonumber\\
&\le & \int_{\Bbb{S}^{n-1}}\int_{0}^{r}e^{\Phi}r^{n-1}dr d\theta,
\end{eqnarray} where $\Phi=f(\theta, r)+f(p)-\frac{2}{r}\int_{0}^{r}f(\theta, s)ds$ and $\rho(\theta)$ denotes the cut-locus radius in the direction of $\theta.$ So, the proof is finished.

\end{proof}

 \subsection{Proof of Theorem  \ref{thmA}}
 
  \begin{proof} Initially, using  the volume form $dV_{exp_{p}(r\theta)}=J^{n-1}(\theta,r)drd\theta$ for $\theta\in S_{p}M,$ for a fixed point $p$, we denote   weighted volume form $dV_f=m_{f}(r,\theta)drd\theta=e^{-f(r,\theta)}J^{n-1}(r,\theta)drd\theta$. Thus, we have 
 
\begin{eqnarray*}
w'(r) + w(r)^2+\frac{1}{n-1}Ric\Big(\dfrac{\partial}{\partial r},\dfrac{\partial }{\partial r}\Big) \leq 0,
\end{eqnarray*} where $w=\frac{J'}{J}.$ Hence, by using the fundamental equation (\ref{eqqem}) we infer

\begin{eqnarray}\label{eq-05}
w'(r) +w(r)^2+\frac{1}{n-1}\left[\frac{1}{m}f'(r)^{2}- f''(r)\right]\leq 0.
\end{eqnarray} Now, one  easily verifies that 

\begin{eqnarray*}
 \dfrac{m'_{f}}{m_f}(r) &=&\dfrac{-f'(r)e^{-f(r)}J^{n-1}(r)+(n-1)e^{-f(r)}J^{(n-2)}(r)J'(r)}{e^{-f(r)}J^{n-1}(r)}\nonumber\\&=&-f'(r)+(n-1)w(r).
\end{eqnarray*} Consequently, $$ \left(\dfrac{m'_{f}}{m_f}(r)\right)'=-f''(r)+(n-1)w'(r).$$ 
Together with  (\ref{eq-05}), this  implies that

 $$\left(\dfrac{m'_{f}}{m_f}(r)\right)'\leq 0.$$ Therefore, upon integrating this from $t_{0}$ to $t$ we obtain

\begin{equation*}
\dfrac{m'_{f}}{m_f}(t)\leq \dfrac{m'_{f}}{m_f}(t_{0}). 
\end{equation*} Integrating once more from $r_{0}$ to $r$ we deduce

\begin{equation*}
\log (m_{f}(r))\leq \log (m_{f}(r_{0}))+\dfrac{m'_{f}}{m_f}(t_{0})(r-r_{0}).
\end{equation*} From this, it follows that

\begin{eqnarray}
\Vol_{f}(B_{p}(r)) \leq be^{cr}
\end{eqnarray} for all $r\geq r_{0},$ where $b$ and $c$ are positive constants. This finishes the proof of the theorem. 
\end{proof}

\subsubsection{Proof of Corollary  \ref{cor1}}
\begin{proof}
The proof of Corollary  \ref{cor1} is standard. Firstly, for a constant $R>1$ sufficient large, we consider a cutoff function $\varphi$ on $B_{p}(R)$ such that $\varphi=1$ on $B_{p}(R-1),$ $\varphi=0$ on $M\setminus B_{p}(R)$ and $|\nabla \varphi|\leq C,$ where $C$ is a constant that does not depend of $R.$ Next, for an arbitrary $\delta>0,$ we set the function $\psi(y)=e^{\alpha r(y)}\varphi (y),$ where $\alpha=-\frac{c+\delta}{2}.$ Therefore, we have 
 
 \begin{eqnarray}
 \label{eq1a}
 |\nabla \psi|^{2}&=&|\alpha e^{\alpha r}\varphi \nabla r+e^{\alpha r}\nabla \varphi|^{2}\nonumber\\&\leq & e^{2\alpha r}(|\alpha| \varphi + |\nabla \varphi|)^{2}.
 \end{eqnarray} Now, we recall that  $$(x+y)^{2}\leq (1+\varepsilon)x^{2}+\left(\frac{1+\varepsilon}{\varepsilon}\right) y^{2},$$ for any $x,$ $y$ and $\varepsilon$ positive. Using this fact into (\ref{eq1a}) we get
 
 \begin{eqnarray}
 \label{eq2a1}
  |\nabla \psi|^{2}&\leq&e^{2\alpha r}\left[ (1+\varepsilon)(\alpha \varphi)^{2}+\left(\frac{1+\varepsilon}{\varepsilon}\right)|\nabla \varphi|^{2}\right].
 \end{eqnarray} Next, we recall that
 
 $$\lambda_{1}(\Delta_{f})=\inf_{\varphi\in C_{0}^{\infty}(M)}\frac{\int_{M}|\nabla \varphi|^{2} d\mu}{\int_{M}\varphi^{2}d\mu},$$ where $d\mu=e^{-f}dv.$ Hence, (\ref{eq2a1}) allows us to infer
  
 \begin{eqnarray}
\lambda_{1}(\Delta_{f})\leq (1+\varepsilon)\alpha^{2}+\left(\frac{1+\varepsilon}{\varepsilon}\right) \frac{\int_{M}e^{2\alpha r} |\nabla \varphi|^{2}d\mu}{\int_{M}e^{2\alpha r}\varphi^{2}d\mu}.
 \end{eqnarray}
 
Proceeding, we remember that $|\nabla \varphi|^{2}\leq C^{2}$ and $0\leq \varphi \leq 1$ on $B_{p}(R)\setminus B_{p}(R-1).$ Thus,  one easily verifies that
 
 \begin{eqnarray}
 \label{1as1}
\frac{\int_{M}e^{2\alpha r} |\nabla \varphi|^{2}d\mu}{\int_{M}e^{2\alpha r}\varphi^{2}d\mu}&=&   \frac{\int_{B_{p}(R)\setminus B_{p}(R-1)}e^{2\alpha r} |\nabla \varphi|^{2}d\mu}{\int_{M}e^{2\alpha r}\varphi^{2}d\mu} \nonumber\\&\leq & C^{2}  \frac{\int_{B_{p}(R)\setminus B_{p}(R-1)}e^{2\alpha r} d\mu}{\int_{B_{p}(1)}e^{2\alpha r} d\mu}\nonumber\\&\leq & C^{2} \frac{e^{2\alpha R}\Vol_{f}(B_{p}(R))}{\Vol_{f}(B_{p}(1))}.
 \end{eqnarray} Thereby, it follows by the proof of Theorem \ref{thmA} that 
 $$\frac{\Vol_{f}(B_{p}(R))}{\Vol_{f}(B_{p}(r_{0}))}\leq b e^{c(R-r_{0})},$$ where $p$ is fixed on $M^n.$ 
 Substituting this into (\ref{1as1})  with $r_{0}=1$ yields 
 
 \begin{eqnarray}
 \label{90aa}
 \frac{\int_{M}e^{2\alpha r} |\nabla \varphi|^{2}d\mu}{\int_{M} e^{2\alpha r}\varphi^{2}d\mu}&\leq &  C^{2} b e^{2\alpha R}e^{c(R-1)}\nonumber\\&=&C^{2}b e^{ -c} e^{(2\alpha+c)R}.
 \end{eqnarray} Next, by $2\alpha+c<0$, the right hand side of (\ref{90aa}) converges to zero when $R$ goes to $\infty.$ Consequently, we have
 
$$\lambda_{1}(\Delta_{f})\leq (1+\varepsilon)\alpha^{2},$$ for all $\varepsilon>0$ and $\delta>0,$ and  by  $\alpha=-\frac{c+\delta}{2}$ we obtain

$$\lambda_{1}(\Delta_{f})\leq (1+\varepsilon)\frac{(c+\delta)^{2}}{4},$$ for all $\varepsilon>0$ and $\delta>0.$  Hence,

$$\lambda_{1}(\Delta_{f})\leq \frac{c^{2}}{4}.$$ So, the proof is finished.
 
 \end{proof}

\subsection{Proof of Theorem  \ref{thm-2}}

\begin{proof} We take an approach similar to the one in \cite{MW} and \cite{MW2}.
To begin with, we already know from the proof of Theorem \ref{thm-1} that

  \begin{equation}\label{jacobi field}
  w'(\theta,r)+w^2(\theta,r)+\frac{1}{n-1}Ric\Big(\frac{\partial}{\partial r},\frac{\partial}{\partial r}\Big)\leq 0,
  \end{equation} where $w(\theta,r)=\frac{J'}{J}(\theta,r).$ Next, it is easy to check that $$\frac{u''}{u}=\Big(\frac{u'}{u}\Big)'+\Big(\frac{u'}{u}\Big)^{2}.$$ Jointly with the fundamental equation (\ref{quasiforu}), this  yields

\begin{eqnarray}\label{thm2-eq9}
  w'(r)+w^2(r)+\frac{m}{n-1}\left(\frac{u'}{u}\right)'(r)+\frac{m}{n-1}\left(\frac{u'}{u}\right)^{2}(r)+\lambda\leq 0,
\end{eqnarray} where we have omitted the dependence of $\theta.$ Upon integrating (\ref{thm2-eq9}) from $1$ to $r\ge 1$ we obtain

\begin{equation*}
w(r)+\int_{1}^{r}w^{2}(t)dt+\frac{m}{n-1}\left(\frac{u'}{u}\right)(r)+\lambda r\leq C,
\end{equation*} where $C$ is a positive constant independent of $r.$ Now, by using (\ref{1a}) we arrive at

\begin{equation}
\label{k1a}
w(r)+\int_{1}^{r}w^{2}(t)dt\leq -\lambda r+C+\frac{1}{n-1}f' (r).
\end{equation}

In order to proceed, we recall that Wang \cite{Wang} proved that if $\lambda\leq 0,$ then $R\ge \lambda n.$ This combined with (\ref{fundamentalequationforu}) allows to infer
 
 \begin{equation}\label{thm2-eq11}
 |f' (r)|^{2}\leq -\frac{m^{2}}{m-1}\lambda.
 \end{equation} Substituting (\ref{thm2-eq11}) into (\ref{k1a}) gives

 \begin{equation}
\label{k3a}
w(r)+\int_{1}^{r}w^2(t)dt\leq -\lambda r+C+\frac{1}{n-1}\sqrt{-\frac{m^{2}}{m-1}\lambda}.
\end{equation}

One verifies by Cauchy-Schwarz inequality that

$$\frac{1}{r}\left(\int_{1}^{r}w(t)dt\right)^{2}\leq \frac{1}{r-1}\left(\int_{1}^{r}w(t)dt\right)^{2}\leq \int_{1}^{r}w^{2}(t)dt.$$ Hence, from (\ref{k3a}) we deduce

\begin{equation}
\label{k4a}
w(r)+\frac{1}{r}\left(\int_{1}^{r}w(t)dt\right)^{2}\leq -\lambda r+C+\frac{1}{n-1}\sqrt{-\frac{m^{2}}{m-1}\lambda}.
\end{equation} We now claim that for any $r\ge 1$, it holds that

\begin{equation}
\label{c1a}
\int_{1}^{r}w(t)dt\leq \left(\sqrt{C+\frac{1}{n-1}\sqrt{-\frac{m^{2}}{m-1}\lambda}-\lambda}\right)  r.
\end{equation} In order to prove this we set $$h(r)=b r-\int_{1}^{r}w(t)dt,$$ where $b=\sqrt{C+\frac{1}{n-1}\sqrt{-\frac{m^{2}}{m-1}\lambda}-\lambda}.$ Therefore, it suffices to prove that $h(r)\ge 0$ for all $r\ge 1.$ Now, we argue by contradiction. Assume that $h$ does not remain nonnegative for all $r\ge 1$. Noticing that $h(1)>0$, we may let $\alpha>1$ be the first number such that $h(\alpha)=0.$ It  follows that $0=h(\alpha)=b \alpha-\int_{1}^{\alpha}w(t)dt,$ that is,  $$\int_{1}^{\alpha}w(t)dt=b \alpha.$$ Substituting this into (\ref{k4a}) we obtain

$$w(\alpha)+\left(C+\frac{1}{n-1}\sqrt{-\frac{m^{2}}{m-1}\lambda}-\lambda\right)\alpha\leq -\lambda \alpha+C+\frac{1}{n-1}\sqrt{-\frac{m^{2}}{m-1}\lambda}.$$ Since $\alpha>1,$ we conclude that $w(\alpha)<0.$ Hence, we deduce

$$h'(\alpha)=\sqrt{C+\frac{1}{n-1}\sqrt{-\frac{m^{2}}{m-1}\lambda}-\lambda}-w(\alpha)>0. $$ This implies the existence of a small enough $\varepsilon>0$ such that $h(\alpha-\varepsilon)<h(\alpha)=0,$ which  leads to a contradiction with the choice of $\alpha$ and therefore we have proved the claim.

Proceeding, we have from (\ref{c1a}) that 

\begin{equation*}
\int_{1}^{r}\frac{J'}{J}(t)dt \leq \left(\sqrt{C+\frac{1}{n-1}\sqrt{-\frac{m^{2}}{m-1}\lambda}-\lambda}\right)  r.
\end{equation*} Consequently,

\begin{equation}
\log (J(r))\leq \left(\sqrt{C+\frac{1}{n-1}\sqrt{-\frac{m^{2}}{m-1}\lambda}-\lambda}\right)  r +\log (J(1))
\end{equation} for all $r\ge 1.$ Thus, we have $$Vol(B_{p}(r))\leq a e^{b r},$$ where $a$ and $b=\left(\sqrt{C+\frac{1}{n-1}\sqrt{-\frac{m^{2}}{m-1}\lambda}-\lambda}\right)$ are positive constants. So, the proof of  the volume growth estimate is finished. 

\end{proof}

\end{document}